\documentclass[12pt,reqno]{amsart}
\usepackage{amsmath,amssymb,amsthm,calc,verbatim,enumitem,tikz,url,hyperref,mathrsfs,cite,fullpage,mathtools}
\usepackage{bbm}
\usepackage{setspace}
\usepackage{latexsym}
\usepackage{color}
\usepackage{enumitem}
\usepackage[right]{lineno}
\usepackage{dsfont}
\usepackage{caption}
\usepackage[ruled]{algorithm2e} 	%Option: linesnumberedhidden
\SetKwInput{KwInput}{Input}                % Set the Input
\SetKwInput{KwOutput}{Output}              % set the Output

\usepackage{theoremref}

\makeatletter
\newcommand{\mylabel}[2]{#2\def\@currentlabel{#2}\label{#1}}
\makeatother

\renewcommand\labelenumi{(\arabic{enumi})}
\renewcommand\theenumi\labelenumi

\newcommand{\arrow}{\rightarrow}

\renewcommand{\Pr}{\mathbb{P}}

\newcommand{\mm}{m_2(K_r,C_{\ell})}

\newcommand{\cE}{\ensuremath{\mathcal{E}}}
\newcommand{\cG}{\ensuremath{\mathcal{G}}}
\newcommand{\cF}{\ensuremath{\mathcal{F}}}
\newcommand{\cH}{\ensuremath{\mathcal{H}}}
\newcommand{\cI}{\ensuremath{\mathcal{I}}}
\newcommand{\cJ}{\ensuremath{\mathcal{J}}}

\newcommand{\cR}{\ensuremath{\mathcal{R}}}
\newcommand{\eps}{\ensuremath{\varepsilon}}

\newcommand{\N}{\ensuremath{\mathbb{N}}}

\newcommand{\se}{\ensuremath{\subseteq}}
\newcommand{\sm}{\ensuremath{\setminus}}
\newcommand{\nto}{\ensuremath{\nrightarrow}}
\newcommand{\wt}{\ensuremath{\widetilde}}

\newcommand{\hypertree}{\textsc{Hypertree}} % Maybe we give this some name, like hypertree, otherwise comment this one out and change references to the algorithm accordingly.  
\newcommand{\flower}{\textsc{Flower}}
\newcommand{\graph}[1]{\textbf{G}\ensuremath{(#1)}} 	% can change to G() if we want to.
\newcommand{\Crit}{\ensuremath{\mathrm{Crit}}}
\newcommand{\Out}{\ensuremath{\mathrm{Out}}}

\newcommand{\oldqed}{}
\def\endofFact{\hfill\scalebox{.6}{$\Box$}}
 
\newtheorem{theorem}{Theorem}

\newtheorem{lemma}[theorem]{Lemma}

\newtheorem{obs}[theorem]{Observation}
\newtheorem{claim}[theorem]{Claim}

\newtheorem{definition}[theorem]{Definition}

\theoremstyle{definition}

\newtheoremstyle{case}{}{}{\normalfont}{}{\itshape}{\normalfont:}{ }{}

\theoremstyle{case}

\definecolor{ao}{rgb}{0.0, 0.5, 0.0}

\numberwithin{theorem}{section}

\begin{document}

\title{Asymmetric Ramsey Properties of Random Graphs for Cliques and Cycles}
\author{Anita Liebenau \and Let\'icia Mattos \and Walner Mendon\c{c}a \and Jozef Skokan}

\address{UNSW Sydney, School of Mathematics and Statistics, Sydney NSW 2052, Australia (A.~Liebenau)}\email{a.liebenau@unsw.edu.au}

\address{IMPA, Estrada Dona Castorina 110, Jardim Bot\^anico, Rio de Janeiro, RJ, Brazil (L.~Mattos and W.~Mendon\c{c}a)} \email{\{\,leticiamat\,|\,walner\,\}@impa.br}

\address{LSE, Department of Mathematics, Houghton Street, London WC2A 2AE, UK (J.~Skokan)}\email{j.skokan@lse.ac.uk}

\thanks{ The first author was supported by an ARC DECRA Fellowship grant DE170100789.
The second and third authors were partially supported by CAPES}

\begin{abstract}
We say that $G \to (F,H)$ if, in every edge colouring $c: E(G) \to \{1,2\}$, we can find either a $1$-coloured copy of $F$ or a $2$-coloured copy of $H$. 
The well-known Kohayakawa--Kreuter conjecture states that the threshold for the property $G(n,p) \to (F,H)$ is equal to $n^{-1/m_{2}(F,H)}$, where $m_{2}(F,H)$ is given by
\begin{align*}
m_{2}(F,H):= \max \left\{\dfrac{e(J)}{v(J)-2+1/m_2(H)} : J \subseteq F, e(J)\ge 1 \right\}.
\end{align*}
In this paper, we show the $0$-statement of the Kohayakawa--Kreuter conjecture for every pair of cycles and cliques.
\end{abstract}

\maketitle

\section{Introduction}

We say that a graph $G$ is a {\it Ramsey graph for the pair of graphs} $(F,H)$ if, in every edge colouring $c: E(G) \to \{1,2\}$, we can find either a $1$-coloured copy of $F$ or a $2$-coloured copy of $H$. 
We write $G \arrow (F,H)$ if $G$ is Ramsey for $(F,H)$, and $G \nto (F,H)$ otherwise.
It follows from Ramsey's Theorem~\cite{Ram} that, for each pair of graphs $(F,H)$, there exists a graph $G$ such that $G\arrow(F,H)$.

The study of whether or not the binomial random graph $G(n,p)$ is Ramsey for a symmetric pair of graphs was initiated by Frankl and R{\"o}dl~\cite{franklrodl}, and {\L}uczak, Ruciński, and Voigt~\cite{luczak1992ramsey}.
They  showed that the probability threshold for having
$G(n,p) \to (K_3,K_3)$ is of order $n^{-1/2}$.
In 1995, R\"odl and Ruci\'{n}ski~\cite{rodl1993lower,rodl1995threshold} determined the probability threshold for $G(n,p) \to (F,F)$ for almost all non-empty graphs $F$.
They showed that, if $F$ has a component which is not a
a star or a path of length three,
then the threshold is of order $n^{-1/m_{2}(F)}$, where 
\begin{align*}
m_{2}(F):=\max\left\{ \frac{e(J)-1}{v(J)-2} : J \subseteq F, v(J) \geq 3 \right \}.
\end{align*}
The parameter $m_{2}(F)$ is called the $m_{2}$-density of the graph $F$. 
Here, $v(J)$ and $e(J)$ denote the size of the vertex set and of the edge set of the graph $J$, respectively.
The remaining cases were addressed subsequently by Friedgut and Krivelevich~\cite{friedgut2000sharp}.

A natural generalisation of this problem is to determine a threshold function $p(F,H)$ for the property $G(n,p)\to (F,H)$, for any asymmetric pair of graphs $(F,H)$.
This problem was posed in 1997 by Kohayakawa and Kreuter~\cite{kohayakawa1997threshold}, who 
proved that $p(C_{\ell},C_{k})=\Theta(n^{1-\ell((\ell-1)k)^{-1}})$ for any pair of cycles $(C_{\ell},C_{k})$ with $k \ge \ell \ge 3$.
In the same paper, they conjectured that $p(F,H)=\Theta(n^{-1/m_{2}(F,H)})$, where
\begin{align*}
m_2(F,H) := \max \left\{\dfrac{e(J)}{v(J)-2+1/m_2(H)} : J \subseteq F, e(J)\ge 1 \right\},
\end{align*}
for any pair of graphs such that $m_{2}(F)\ge m_{2}(H)\ge 1$.
Since the Kohayakawa--Kreuter conjecture was posed, there have been many attempts to solve it (see, for example,~\cite{marciniszyn2009asymmetric,kohayakawa2014upper,gugelmann2017symmetric}).
In a recent breakthrough, Mousset, Nenadov and Samotij~\cite{mousset2018towards} showed that 
$p(F,H)= O(n^{-1/m_{2}(F,H)})$ whenever 
$m_{2}(F)\ge m_{2}(H)\ge 1$,
the so-called \emph{$1$-statement}.
In contrast, much less is known about the \emph{$0$-statement}, 
that is, the statement that $p(F,H)= \Omega(n^{-1/m_{2}(F,H)})$ whenever $m_{2}(F)\ge m_{2}(H)\ge 1$.
One possible reason for that is that the $0$-statement seems to depend on the structural behaviour of Ramsey graphs.

As far as we know, the $0$-statement is only proved for two types of pairs of graphs.
Kohayakawa and Kreuter~\cite{kohayakawa1997threshold} established the $0$-statement for all pairs of cycles while Marciniszyn, Skokan, Sp{\"o}hel and Steger~\cite{marciniszyn2009asymmetric} addressed all pairs of cliques.
{\setstretch{1.116}

In this paper, we show that the $0$-statement holds for any pair of cliques and cycles.
This is the first $0$-statement result for different types of graphs.

\begin{theorem}\thlabel{thm:main}
For all $\ell,r \geq 4$ there exists $c>0$ such that, if $p = p(n)\leq cn^{-1/m_2(K_r,C_{\ell})}$, then
\begin{align*}
\lim \limits_{n \to \infty} \Pr \big[G(n,p)\arrow(K_r,C_\ell)\big] = 0.
\end{align*}
\end{theorem}
Combining~\thref{thm:main} with the results of~\cite{kohayakawa1997threshold},~\cite{marciniszyn2009asymmetric} and~\cite{mousset2018towards}, we establish the Kohayakawa--Kreuter conjecture for any pair of cycles and cliques with at least $3$ vertices.
We remark that we do need the assumption $\ell,r \ge 4$ in our
proof, so our result does not imply the earlier results
involving $K_{3}$.

The main tool behind the proof of Theorem~\ref{thm:main} is a structural characterisation of Ramsey graphs for the pair $(K_{r},C_{\ell})$ via a `container type' argument (see Theorem~\ref{thm:maintech}), which is a rephrasing of the idea used in previous works. 
Roughly speaking, we find a family $\mathcal{I}$ of graphs with the following properties: 
(a) $|\mathcal{I}|$ is small;
(b) for every graph $G$ with $G \to (K_{r},C_{\ell})$ there exists $I \in \mathcal{I}$ such that $I \subseteq G$; and 
(c) for each $I \in \mathcal{I}$, either $I$ is small and dense or very structured.
We provide the details in Section~\ref{overview}.

The rest of the paper is organised as follows.
In Section 2, we prove~\thref{thm:main}; 
in Section 3, we provide the main technical lemmas of this paper;
in Section 4, we prove some structural lemmas about Ramsey graphs; in Section 5, we describe the algorithms used to prove our main technical theorem (see Theorem~\ref{thm:maintech}); 
finally, in Section 6, we do a careful analysis of these algorithms.
In the appendix, we provide some simple calculations involving $m_{2}$-densities, for completeness.

}

\section{The main technical result }\label{overview}

In this section, we present the main technical result of this paper and deduce~\thref{thm:main} from it. 
In order to state this result, we need a little notation.
For a graph $G$, define $\lambda(G)$ by 
\begin{equation*}
	\lambda(G) =v(G)-\frac{e(G)}{\mm}.
\end{equation*}
For any positive real numbers $M,\varepsilon$ and any positive integer $n$, define
\begin{align}\label{Jsets}
	\mathcal{J}_1(\varepsilon)=\{G: \lambda(G)\le-\varepsilon \} 
	\qquad \text{and} \qquad 
	\mathcal{J}_2(M,n)=\{G: \lambda(G)\le M \text{ and } e(G)\ge
	\log n \},
\end{align}
where the logarithm is in base 2. Finally, for any natural numbers $r,\ell$ and $n$, let 
\begin{align*}
	\mathcal{R}_n(K_r,C_{\ell})=
	\big\{ G : v(G) = n \, \text{ and } \, G \to (K_r,C_\ell) 
	\big\}.
\end{align*}
When $r$ and $\ell$ are clear from context, we write $\cR_n$ for $\mathcal{R}_n(K_r,C_{\ell}).$ 
In addition, we set 
\begin{align*}
	\cR(K_{r},C_{\ell})= 
	\bigcup_{n \in \mathbb{N}} \cR_{n}(K_{r},C_{\ell}).
\end{align*}
The connection between $\lambda$, $\cJ_1(\eps)$, $\cJ_2(M,n)$ and $\cR_n$ is contextualised in the next theorem. 
\begin{theorem}\thlabel{thm:maintech}
	For any integers $r,\ell \ge 4$, 
	there exist positive constants $M=M(r,\ell)$ 
	and $\varepsilon=\varepsilon(r,\ell)$ such that 
	the following holds. 
	For every $n \in \mathbb{N}$, there exists a function 
	$f: \mathcal{R}_n(K_r,C_{\ell}) \to 
	\cJ_1(\varepsilon)\cup \cJ_2(M,n)$ 
	such that $f(G)\subseteq G$ for all $G \in \mathcal{R}_n$ and
	\begin{align*}
		|f(\cR_{n})|\le (\log n)^{M}.
	\end{align*}
\end{theorem}

In the language of hypergraph containers~\cite{BMS,ST},
~\thref{thm:maintech} provides a relatively small collection 
$f(\cR_{n})$ of \emph{fingerprints}.
Additionally to $|f(\cR_{n})|$ being small, each graph $f(G)$ 
either 
has a \emph{very} small value of $\lambda$ (negative, and bounded 
away from 0), or a fairly small (though possibly positive) value 
of $\lambda$ and is \emph{very} large. 
To obtain such a collection and the function $f$ 
in~\thref{thm:maintech}, we employ an algorithm adapted 
from~\cite{kreuter}. 
{\setstretch{1.09}

\thref{thm:main} is easily deduced from~\thref{thm:maintech}.
The proof of~\thref{thm:maintech} is given in the next four sections.

\begin{proof}[Proof of~\thref{thm:main}] 
	Given $r,\ell\ge 4$, let $M$ and $\varepsilon$ be positive 
	constants given by~\thref{thm:maintech} and set 
	$c=2^{-2M}$. 
	For each $n \in \mathbb{N}$, let $p=p(n) \le c n^{-1/\mm}$. 
	Let $f$ be the function given by~\thref{thm:maintech}, let 
	$\Gamma \sim G(n,p)$ and suppose that 
	$\Gamma\in \mathcal{R}_{n}.$ 
	Then, $f(\Gamma) \subseteq \Gamma$ and 
	$f(\Gamma) \in \cJ_1(\varepsilon)\cup 
	\cJ_2(M,n)$.
	Let $\cI_{1}:=f(\cR_{n})\cap \cJ_{1}(\varepsilon)$ and 
	$\cI_{2}:=f(\cR_{n})\cap \cJ_{2}(M,n)$.
	Thus, 
	\begin{equation} \label{boundmainprob}
		\Pr \big(\Gamma\to (K_r,C_{\ell})\big)\le
		\Pr \big(F \subseteq \Gamma \text{ for some } F \in 
		\cI_{1} \cup \cI_{2}\big).
	\end{equation}
	Since $\lambda(F)\le-\varepsilon$ for each 
	$F \in \mathcal{I}_{1}$ and $c\le 1$, we have
	\begin{equation}\label{boundI1}
		\mathbb{P}(F \subseteq \Gamma)\le 
		n^{v(F)}p^{e(F)} 
		\le c^{e(F)}n^{\lambda(F)} 
		\le n^{-\varepsilon}
	\end{equation}
	for every $F \in \cI_{1}$.
	Similarly, we have
	\begin{equation}\label{boundI2}
		\mathbb{P}(F \subseteq \Gamma)\le c^{e(F)}n^{\lambda(F)} 
		\le  n^{-M}
	\end{equation}
	for every $F \in \mathcal{I}_2,$ 
	as $\lambda(F)\le M$ and $e(F)\ge \log n$ for each 
	$F \in \mathcal{I}_{2}$, and by our choice of $c$.
	Applying the union bound to~\eqref{boundmainprob} and 
	using~\eqref{boundI1} and~\eqref{boundI2}, we obtain that
	\begin{align*}
		\mathbb{P}\big(G(n,p)\to (K_r,C_{\ell})\big)\le 
		(\log n)^{M} \cdot (n^{-\varepsilon}+n^{-M}), 
	\end{align*} 
	since 
	$|\mathcal{I}_{1}\cup \mathcal{I}_{2}|\le (\log n)^{M}$.
	As the expression on the right hand side tends to 0 as 
	$n\to\infty$, this implies the theorem.
\end{proof}

}

\section{Proof of Theorem~\ref{thm:maintech}}\label{proofofmain}

{\setstretch{1.12} In this section, we state the main technical lemmas of this paper, and deduce~\thref{thm:maintech} from them. 
We also introduce some notation that we use during the proof.

Let $r,\ell$ be positive integers. 
We follow the approach by~\cite{kohayakawa1997threshold} and 
\cite{kreuter} and bring our problem into the hypergraph setting. 
Given a graph $G=(V,E)$, let $\cG_{r,\ell}(G)$ be the hypergraph 
on the edge set of $G$ whose hyperedges correspond to the copies of $K_{r}$ and $C_{\ell}$ in $G$.
We suppress $G$, $r$ and $\ell$ from the notation whenever they are clear from context. 
Define
\begin{align}\label{hyperedges}
	\cE_1^{r}(\cG) &= \{ E(F): F \cong K_r, F \subseteq G \} 
	\text{ and } 
	\cE_2^{\ell}(\cG) = \{ E(F): F \cong C_{\ell}, 
	F \subseteq G \}.
\end{align}
Analogously, if $\cH \se \cG_{r,\ell}(G)$, then we set $\cE_{1}(\cH):=\cE_{1}(\cG)\cap E(\cH)$ and $\cE_{2}(\cH):=\cE_{2}(\cG)\cap E(\cH)$.

The reason why we deal with $\cG_{r,\ell}(G)$ instead of $G$ is as follows. 
In order to build our fingerprints algorithmically, 
we would like to deal with a subgraph $H$ of $G$ which has the two following properties: (1) every edge $e \in E(H)$ is contained in an $\ell$-cycle; (2) for {\em every} copy $C$ of $C_{\ell}$ in $H$ end every $e\in C$, there exists a copy $K$ of $K_{r}$ such that $E(K)\cap E(C) =\{e\}$.
These properties would allow us to build the fingerprint of $G$ algorithmically
by attaching either a copy of $K_{r}$ or a copy of $C_{\ell}$ to the current graph at each step.
However, such graph a $H$ might not exist.
Even if $H$ is minimal (with respect to subgraph containment) for $H \to (K_r,C_{\ell})$,
we can only deduce that 
for every $e \in E(H)$ there exist $K \cong K_{r}$ and $C \cong C_{\ell}$ in $H$ such that $E(K)\cap E(C) =\{e\}$.
But this does not directly imply that property (2) holds.
We can overcome this problem by considering subhypergraphs of $\cG_{r,\ell}(G)$ which are $\star$-critical.
This property was first considered in~\cite{kreuter}.

\begin{definition}[$\star$-critical]\thlabel{starcritical} 
Let $\cE_{1}$ and $\cE_{2}$ be two families of sets on a vertex set.
We say that a hypergraph $\cH = \cE_{1} \cup \cE_{2}$ is $\star$-critical with respect to $(\cE_{1},\cE_{2})$ if the following two properties hold.
For each $e \in V(\cH)$, there exists a hyperedge $F \in \cE_{2}$ such that $e \in F$;
and for each $F \in \cE_{2}$ and each $e \in F$, 
there exists a hyperedge $E \in\cE_{1}$ such that $E \cap F = \{e\}$.
When $\cE_{1}$ and $\cE_{2}$ are clear from context, we say that the hypergraph $\cH$ is $\star$-critical.
\end{definition}

Let $\Crit_{r,\ell}(G)$ be the set of all $\star$-critical subhypergraphs of $\cG_{r,\ell}(G)$.
The next lemma shows that if $G \to (K_r,C_{\ell})$, then
there are subhypergraphs of $\cG_{r,\ell}(G)$ which are $\star$-critical.
This is a direct consequence of~\thref{ramsey-crit}, 
and we prove it in Section~\ref{sec:structural}.
\begin{lemma}\thlabel{GRamseyhascritical}
	Let $r, \ell \geq 4$ be integers.
	If $G \to (K_{r},C_{\ell})$, then $\Crit_{r,\ell}(G)\neq 
	\emptyset$.
\end{lemma}

Given a hypergraph $\cH\se \cG_{r,\ell}(G)$, 
we define the \emph{underlying graph of $\mathcal{H}$},
denoted by $\graph{\mathcal{H}}$, 
to be the subgraph of $G$ whose edge set is 
$\bigcup_{E \in E(\cH)} E$. 
The following lemma is central to our proof. We prove it in Section~\ref{sec:structural}. }

\begin{lemma}\thlabel{dens-critic}
	Let $r, \ell \geq 4$ be integers. There exists 
	$\eps=\eps(r,\ell)>0$ such that the following holds.
	If $\cH \in \Crit_{r,\ell}(H)$,
	then $\lambda\big(\graph{\mathcal{H}}\big)\le -\eps$.
\end{lemma}

In order to find the function 
$f: \mathcal{R}_n(K_r,C_{\ell}) \to \cJ_1(\varepsilon)\cup \cJ_2(M,n)$ in~\thref{thm:maintech}, 
we define an algorithm \hypertree\ in Section~\ref{sec:algos}. 
For each $G \in \cR_n(K_r,C_{\ell})$, this algorithm takes some hypergraph
$\cH \in \Crit_{r,\ell}(G)$ as input and 
creates a pair $(\cH_{T},D_{T})$ as output, 
where $\cH_{T} \subseteq \cH$ and $D_{T} \subseteq \mathbb{N}$.
The fingerprint $f(G)$ will be $\graph{\cH_{T}}$, 
and the auxiliary set $D_{T}$ will help us to count all the possible 
outputs of \hypertree\ and to ensure that $\graph{\cH_{T}}$ belongs to 
$\cJ_{1} \cup \cJ_{2}$.
For the detailed description of \hypertree, we refer the reader to Section~\ref{sec:algos}. 
Let \hypertree$(\cH)$ denote the execution of \hypertree\ on input 
$\cH$.
Its basic properties are given by the next lemma.

\begin{lemma}\thlabel{hypertreeBasics}
	Let $n,r,\ell\ge 4$ be integers and 
	$G$ be a graph on $n$ vertices.
	For any hypergraph $\cH\in\Crit_{r,\ell}(G)$,
	\hypertree$(\cH)$ generates a sequence of hypergraphs 
	$\mathcal{H}_{0}\se \ldots \se \mathcal{H}_{T}\se 
	\mathcal{H}$ and sets $D_0\se\ldots\se D_T$ 
	for which the following holds.
	\begin{enumerate}[label=(\alph*)]
	\item[\mylabel{hypertreeBasics-E0}{$(a)$}]
		$\cH_0=\{E\}$, for some $E\in \cE_{1}(\cH)$; 
		that is, the underlying graph of 
		$\cH_0$ is a copy of $K_r$ in $G$; 
	\item[\mylabel{hypertreeBasics-vertices}{$(b)$}]
		$v(\mathcal{H}_{0})< v(\mathcal{H}_{1})< \ldots 				< v(\mathcal{H}_{T})$; 
	\item[\mylabel{hypertreeBasics-stopping}{$(c)$}]
		$T$ is the smallest integer such that 
		$\lambda(G_T)\le -\eps$ or 
		$T \ge \log n$, 
		where $G_T=\graph{\mathcal{H}_T}$ and 
		$\eps$ is the constant given by~\thref{dens-critic};
	\item[\mylabel{hypertreeBasics-return}{$(d)$}]
		\hypertree$(\cH)$ returns the pair $(\cH_T,D_T)$.
	\end{enumerate}
\end{lemma}

Our next lemma is the most important property of the \hypertree\ algorithm (for a proof, see Section~\ref{sec:analysis}).
We shall use it together with~\thref{hypertreeBasics} to 
deduce that the underlying graph given by the output of \hypertree$(\cH)$ belongs to 
$\cJ_{1}(\eps) \cup \cJ_{2}(M,v(G))$ 
whenever $\cH$ is a hypergraph in $\Crit_{r,\ell}(G)$,
where $M=M(r,\ell)>0$.
As the underlying graph of the output hypergraph is 
a subgraph of $G$, this will establish the existence of a function 
$f: \mathcal{R}_n(K_r,C_{\ell}) \to \cJ_{1}(\eps)
\cup \cJ_{2}(M,n)$ such that $f(G)\subseteq G$.

\begin{lemma}\thlabel{deg-lambda} 
	For all integers $r,\ell\ge 4$, there exists 
	$\delta = \delta(r,\ell) >0$ such that the following holds.
	For any graph $G$ and any hypergraph 
	$\cH \in \Crit_{r,\ell}(G)$, the sequence 
	$(\mathcal{H}_i,D_i)_{i=0}^{T}$ generated by 
	\hypertree$(\cH)$ satisfies
	
	\begin{enumerate}
		\item[\mylabel{item:1@deg-lambda}{$(1)$}] 
    	$\lambda (G_{i}) = \lambda (G_{i-1})$ 
		for all $i\not\in D_T$, and 
		
		\item[\mylabel{item:2@deg-lambda}{$(2)$}] 
		$\lambda (G_{i}) \leq 
		\lambda (G_{i-1})- \delta$ for all $i\in D_T$, 
	\end{enumerate}
	where $G_{i}=\graph{\cH_{i}}$ for each 
	$i \in \{0,\ldots,T\}$.
	\end{lemma}

We say that the $i$-th step of \hypertree$(\cH)$ is \emph{degenerate} if $i\in D_T$, and \emph{non-degenerate} otherwise.
The stopping conditions on $\lambda$ (see~\thref{hypertreeBasics}\ref{hypertreeBasics-stopping}) combined with~\thref{deg-lambda} imply that \hypertree$(\cH)$ must have a bounded number of degenerate steps.
This will be essential to show that the set of all possible outputs given by \hypertree\ is at most polylogarithmic in $n$ when applied over $\bigcup_{v(G)=n} \Crit_{r,\ell}(G)$.
To be more precise, for each $n,r,\ell\ge4$, 
{\setstretch{1.1} consider the family of non-isomorphic graphs
	\begin{align*}
		\Out_{r,\ell}(n) = 
		\bigcup_{G \,:\, v(G) = n} \big\{ \graph{\cH_T} \ : \ \cH\in \Crit_{r,\ell}(G) \big\},
	\end{align*}	
where $T= T(\cH)$ and $\cH_T$ are the stopping time and the output given by \hypertree$(\cH)$, respectively. 
The next lemma bounds the size of $\Out_{r,\ell}(n)$. 
\begin{lemma}\thlabel{lem:boundOut}
		For all $r,\ell \ge 4$, there exists $C>0$ such that 
		$|\Out_{r,\ell}(n)|\le (\log n)^{C}$, for all $n\in \N$. 
\end{lemma} 
We prove this lemma in Section~\ref{sec:analysis}.
Now, we are ready to prove~\thref{thm:maintech} assuming all the lemmas stated in this section.

\begin{proof}[Proof of~\thref{thm:maintech}]
	Fix $n,r,\ell \ge 4$.
	For each $G \in \mathcal{R}_{n}(K_{r},C_{\ell})$,
	let $\cH(G)$ be a $\star$-critical hypergraph in
	$\Crit_{r,\ell}(G)$.
	By~\thref{GRamseyhascritical}, such a hypergraph must exist.
	Let $(\cH_{i}(G))_{i=0}^{T}$ be the 
	sequence of hypergraphs generated by \hypertree$(\cH(G))$.
	By~\thref{hypertreeBasics}, the last hypergraph of this 
	sequence, namely $\cH_{T}(G)$, is the hypergraph output by 
	\hypertree$(\cH(G))$.

	Define
	\begin{align*}
		f: \mathcal{R}_{n} &\to \Out_{r,\ell}(n) \\
		G & \mapsto \graph{\cH_{T}(G)}.
	\end{align*}
	As $\cH_{T}(G) \se \cH(G)$ by~\thref{hypertreeBasics}, and 
	$\graph{\cH(G)}\se G$ by construction, we have
	$f(G)\subseteq G$ for each $G \in \cR_{n}$.
	Moreover, 
	by~\thref{hypertreeBasics}\ref{hypertreeBasics-stopping}, 
	we have $\lambda(f(G))\le -\eps$  or $T \ge \log n$.
	In the first case, $f(G)$ belongs to the set of graphs
	\begin{align*}
	\mathcal{J}_1(\eps)=\{H: \lambda(H)\le-\eps \}.
	\end{align*}
	In the second case, we claim that 
	$f(G) \in \cJ_2(C,n)$, where $C = \lambda(K_r)$.
	To see that, first note that the sequence 
	$\big(v(\cH_{i}(G))\big)_{i=0}^{T}$ is 
	strictly increasing 
	by~\thref{hypertreeBasics}\ref{hypertreeBasics-vertices}.
	For simplicity, let $G_{i} = \graph{\cH_{i}(G)}$.
	Since $e(G_{i}) = v(\cH_{i}(G))$ for every 
	$i \in \{0,\ldots,T\}$, 
	we have
	\begin{align}\label{J_2ineq1}
	e(f(G))=v(\cH_{T}(G))\ge T \ge \log n.
	\end{align}
	Moreover, by~\thref{deg-lambda}, we have 
	$ \lambda (G_{i}) \le 
	\lambda (G_{i-1}) \text{ for each }i \in \{0,\ldots,T\}$,
	where $G_{0}\cong K_{r}$ by~\thref{hypertreeBasics}
	\ref{hypertreeBasics-E0}.
	In particular,
	\begin{align}\label{J_2ineq2}
		\lambda(f(G))=\lambda(G_{T})\le 
		\lambda(G_{0}) = \lambda(K_{r}).
	\end{align}
	Together,~\eqref{J_2ineq1} and~\eqref{J_2ineq2} imply that
	$f(G) \in \cJ_2(C,n)$, where $C=\lambda(K_{r})$. 
	This proves our claim.

	Now, it only remains to show that 
	$|f(\cR_{n})|\le (\log n)^{C_{0}}$, for some constant 
	$C_{0}>0$.
	But, this follows directly from~\thref{lem:boundOut}. 
	We finish the proof by setting $M=\max\{C_{0},C\}$.
\end{proof}
}

\section{The structural lemmas}\label{sec:structural}

In this section, we obtain some key structural information
about Ramsey hypergraphs and prove Lemmas~\ref{GRamseyhascritical}
and~\ref{dens-critic}.

Given two families of sets $\cE_{1}$ and $\cE_{2}$ on a vertex set
and a hypergraph $\cH$,
define 
\begin{align*}
\cE_{1}(\cH)= \cE_{1}\cap E(\cH) \qquad \text{and} 
\qquad \cE_{2}(\cH)= \cE_{2}\cap E(\cH).
\end{align*}
We refer to the hyperedges of $\cE_{1}(\cH)$ and $\cE_{1}(\cH)$ as, respectively, hyperedges of type $1$ and $2$.
We say that $\cH$ is Ramsey for $(\cE_1,\cE_2)$, and we write $\cH \to (\cE_1,\cE_2)$,
if the following holds.
For every $2$-colouring $c: V(\cH)\to \{1,2\}$, there exists a hyperedge $E \in \cE_{i}(\cH)$ such that $c(E)=\{i\}$, 
for some $i \in \{1,2\}$. 
Conversely, we write $\mathcal{H}\nto (\cE_1,\cE_2)$ if $\mathcal{H} \to (\cE_1,\cE_2)$ is not satisfied. 
Clearly, if $\cH \subseteq \cF$ and $\cH \to (\cE_1,\cE_2)$, then $\cF \to (\cE_1,\cE_2)$.
Therefore, we may concentrate on the minimal hypergraphs 
$\cH$ that satisfy $\cH\to (\cE_1,\cE_2)$.

We call a hypergraph $\cH$ \emph{Ramsey minimal} with respect to 
$(\cE_1,\cE_2)$ if $\cH \to (\cE_1,\cE_2)$, yet the removal of any hypervertex or hyperedge from $\cH$ destroys this property. 
Minimal Ramsey hypergraphs have the $\star$-critical property, as the next lemma shows.
Its proof follows the same steps as the proof of~\cite[Claim 1]{kreuter}.
A particular case of it can be also found in~\cite[Section 3]{kohayakawa1997threshold}.

\begin{lemma}\thlabel{ramsey-crit}
	Let $\cE_{1}$ and $\cE_{2}$ be two disjoint families of sets 
	on a vertex set.
	If a hypergraph $\cH$ is Ramsey minimal with respect to 
	$(\cE_{1}, \cE_{2})$, then the following holds.
	For each $i \in \{1,2\}$, each hyperedge $E \in \cE_i$, 
	and each hypervertex $e \in E$, there exists a hyperedge
	$F\in\cE_{3-i}$ such that $E \cap F = \{e\}$.
	In particular, $\cH$ is $\star$-critical.
\end{lemma}

\begin{proof}
	Fix any hyperedge $E \in \cE_{i}(\cH)$, 
	for some $i \in \{1,2\}$, and any hypervertex 
	$e \in E$.
	Let $\cH \setminus E$
	be the hypergraph with vertex set $V(\cH)$ and
	hyperedge set $E(\cH)\setminus \{E\}$.
	Consider any colouring $c: V(\cH) \to \{1, 2\}$  
	for which there is no hyperedge of type $j$ in 
	$\cH \setminus E$ coloured $j$ under $c$,
	for all $j \in \{1,2\}$.
	This colouring exists because $\cH$ is Ramsey minimal.
	As $\cH \to (\cE_{1},\cE_{2})$, all the hypervertices in $E$ 
	must be coloured $i$ under $c$.
	Moreover, $E$ is the only monochromatic
	hyperedge under $c$ which has the same colour as its number 
	type.
	Now, let $c': V(\cH) \to \{1,2\}$ 
	be the colouring such that $c'(f) = c(f) \iff f \neq e$
	(recall that $e \in E$).
	As $\cH \to (\cE_{1},\cE_{2})$ and
	$E$ is not monochromatic of colour $i$ under $c'$, there must	exist a hyperedge $F \in \cE_{3-i}(\cH)$ such that 
	$c_E'(F)=3-i$ 
	and $E \cap F = \{e\}$, as required.
\end{proof}

Now we are ready to prove~\thref{GRamseyhascritical}.

\begin{proof}[Proof of~\thref{GRamseyhascritical}]
	If $G \to (K_{r},C_{\ell})$, then 
	$\cG_{r,\ell}(G)\to (\cE_{1},\cE_{2})$, where 
	$\cE_{1}$ and $\cE_{2}$ are defined in~\eqref{hyperedges}.
	Let $\cH$ be an arbitrary Ramsey minimal subhypergraph of
	$\cG_{r,\ell}(G)$. 
	Then $\cH$ is $\star$-critical, by~\thref{ramsey-crit}, and 
	therefore $\cH \in \Crit_{r,\ell}(G)$, as required.
\end{proof}

We now turn to the proof of~\thref{dens-critic}.
In order to prove it, we require some structural information about underlying graphs of $\star$-critical hypergraphs.
This structural information is obtained in~\thref{inde-critic} and,
before stating it, it is worth to point out the following observation.
\begin{obs}\thlabel{degr}
Let $r,\ell \ge 3$.
For any graph $H$ and any hypergraph $\cH \in \Crit_{r,\ell}(H)$, 
we have $d(v)\ge r$ for all $v \in V(\graph{\cH})$.
\end{obs}

In fact, if $\cH$ is $\star$-critical, then for any $e \in E(\graph{\cH})$ there exists $K \in \cE_{1}(\cH)$ and $C \in \cE_{2}(\cH)$
such that $C \cap K = \{e\}$.
As $K$ and $C$ are copies of $K_{r}$ and $C_{\ell}$ contained in $\graph{\cH}$, respectively,
we can easily infer that $d(v)\ge r$ for all $v \in V(\graph{\cH})$.

Now, we need to set some notations. 
For each graph $G$, define
\begin{align}\label{AandB}
	A=A(G)=\{v \in V(G):d(v)=r \} \qquad \text{and} \qquad 
	B=B(G)=\{v \in V(G):d(v)>r \}.
\end{align} 
By~\thref{degr}, $V(G)$ can be partitioned into $V(G) = A \cup B$ whenever $G$ is the underlying graph of a $\star$-critical hypergraph.
Below, we use $N(v)$ to denote the neighbourhood of a vertex $v$ in $G$ and, for each $S \subseteq V(G)$, we write $d_{S}(v)=|N(v)\cap S|$.
Our structural lemma is as follows.

\begin{lemma}\thlabel{inde-critic}
	Let $r \ge 3$ and $\ell \ge 4$ be integers, 
	and $H$ be a graph.
	For any hypergraph $\cH \in \Crit_{r,\ell}(H)$, 
	we have that
	\begin{enumerate}
	\item[$(1)$] $A$ is an independent set in $\graph{\cH}$ and
	\item[$(2)$] $d_{B}(v)\ge r-2$, 
	for all $v \in V(\graph{\cH})$.
	\end{enumerate}
\end{lemma}

\begin{proof}[Proof]	
	First, let us prove item (1). 
	Suppose for a contradiction that there are two adjacent 
	vertices $u,v \in V(\graph{\cH})$ such that $d(u)=d(v)=r$. 
	As $\mathcal{H}$ is $\star$-critical, there exists an 
	$r$-clique $R_1 \in \cE_{1}(\cH)$ and an $\ell$-cycle 
	$C_1 \in \cE_{2}(\cH)$ such that 
	$E(R_1) \cap E(C_1)=\{u,v\}$. 
	Now, let $u_{c}$ be the neighbour
	of $u$ in $V(C_{1})\setminus \{u,v\}$
	and fix any vertex $w \in V(R_{1})\setminus \{u,v\}$.
	Then, we have the following claim.
	\begin{claim}
		There exists an $\ell$-cycle $C_{2}\in \cE_{2}(\cH)$ 
		such that either 
		$\{u_{c},u,w\}\se V(C_{2})$ or $\{v,u,w\}\se V(C_{2})$.
	\end{claim}

	\begin{proof}[Proof of the Claim]
		As $\cH$ is $\star$-critical, there exists 
		an $r$-clique $R_2$ and an $\ell$-cycle $C_2$ such that 
		$E(R_2) \cap E(C_2)=\{u,w\}$. 
		If $R_{2} = R_{1}$, then $C_{2}$ must contain $u_{c}$, 
		as $d(u)=r$. 
		This settles the first part of the claim.
		If $R_{2}\neq R_{1}$, then $R_{2}$ must contain $u_{c}$,
		again because $d(u)=r$.
		Moreover, as $u_{c} \in V(R_{2})$ and $\ell \ge 4$, 
		we cannot have $v \in R_{2}$.
		Together, these imply that 
		$V(R_{2}) = \{u_{c}\} \cup V(R_{1})\setminus \{v\}$.
		As $v$ is the only vertex in $V(R_{1})$ not contained in 
		$R_{2}$, it follows that $v \in C_{2}$.
		This settles the second part.
	\end{proof}

	Let $C_{2}$ be the cycle given by the claim above.
	If $\{u_{c},u,w\} \in V(C_{2})$, 
	then there exists 
	an $r$-clique $R_{2}\in \cE_{1}(\cH)$ such that 
	$E(R_2) \cap E(C_2)=\{u_{c},u\}$. 
	As $V(R_{2})\se N(u)\cup \{u\}\setminus\{w\}$,
	$R_{2}$ has no choice but to contain $v$.
	In particular, this implies that $v$ is a neighbour of 
	$u_{c}$, which gives us a contradiction, as $\ell \ge 4$.
	If $\{v,u,w\} \in V(C_{2})$, 
	then there exists an $r$-clique $R_{2}\in \cE_{1}(\cH)$ 
	such that 
	$E(R_2) \cap E(C_2)=\{u,v\}$. 
	As $V(R_{2})\se N(u)\cup \{u\}\setminus\{w\}$,
	$R_{2}$ has no choice but to contain $u_{c}$.
	In particular, this implies again that $v$ is a neighbour 
	of $u_{c}$, which gives us a contradiction.
	This proves item $(1)$.

	To show item (2), we consider two cases: 
	(i) either $d_{B}(v)=d(v)$, or (ii) $d(v)>d_{B}(v)$.
	In the first case,~\thref{degr} gives us 
	$d_{B}(v)=d(v)\ge r$. 
	In the second case, there is a vertex 
	$u \in N(v)\cap A$ and, since $\cH$ is $\star$-critical, 
	there is also an $r$-clique $R$ in $\graph{\cH}$ such that 
	$\{u,v\} \in E(R)$. 
	As $A$ is an independent set, we must have 
	$V(R)\setminus \{u\} \subseteq B$, 
	which implies that $v$ has least 
	$r-2$ neighbours in $B$. 
\end{proof}

Let $m(G)=e(G)/v(G)$ be the edge density of $G$.
Now, we are ready to prove~\thref{dens-critic}.

\begin{proof}[Proof of~\thref{dens-critic}]
	To simplify the notation, set $G = \graph{\cH}$.
	Observe that $\lambda(G) \le -\eps$ if and only if 
	$\mm^{-1}-m(G)^{-1}\ge \eps e(G)^{-1}$.
	The last inequality holds for some $\eps=\eps(r,\ell)>0$ if 
	there exists $\delta=\delta(r,\ell)>0$ such that 
	$m(G)>\mm+\delta$.
	Thus, we can reduce our problem to finding such $\delta$.
	In order to do so, we shall first bound $e(G)$.

	By~\thref{degr}, the set $V(G)$ can be partitioned into 
	$V(G)=A\cup B$, 
	where $A=A(G)$ and $B=B(G)$ were defined in~\eqref{AandB}.
	Thus, we can write
	\begin{align*}
		2e(G)= \sum \limits_{v \in A} d(v) + 
		\sum \limits_{v \in B} d(v).
	\end{align*}	
	As $d(v)=r$ for all $v \in A$, we have 
	$\sum_{v \in A} d(v)=r|A|$.
	Now, to bound the sum $S=\sum_{v \in B}d(v)$, 
	observe that this sum counts twice each edge inside the set 
	$B$ and counts once each edge across $A$ and $B$.
	By~\thref{inde-critic}$(1)$, we have $e(A,B)=r|A|$, 
	as $A$ is an independent set and $d(v)=r$ for each $v \in A$.
	By~\thref{inde-critic}$(2)$, $d_{B}(v)\ge r-2$ 
	for each $v \in B$.
	Together, these imply that 
	\begin{align*}
		S \ge r|A|+(r-2)|B|.
	\end{align*}	
	Furthermore, 
	\begin{align*}
		S \ge (r+1)|B|,
	\end{align*} 
	as $d(v)\ge r+1$ for each $v \in B$.
	Therefore, 
	\begin{align*}
		2e(G)\ge  
		r |A|+\max \Big\{ r|A|+(r-2)|B|, (r+1)|B| \Big\}.
	\end{align*}
	As $v(G)=|A|+|B|$, we have
	\begin{align*}
		2m(G) & \ge  \max \left\{ 
		\dfrac{2r|A|+(r-2)|B|}{|A|+|B|}, 
		\dfrac{r|A|+(r+1)|B|}{|A|+|B|} \right\} \\
		& = r-2+ \max \big\{ (r+2)x, 3-x \big\},
	\end{align*}
	where $x=|A|/v(G)$. 
	The last expression attains its minimum value when 
	$x=3/(r+3)$, and hence
	\begin{align*}
		m(G) \ge \dfrac{r+1}{2}- \dfrac{3}{2(r+3)}.
	\end{align*}

	A straightforward calculation shows that 
	$m_{2}(K_{r},C_{\ell})=
	\frac{\binom{r}{2}(\ell-1)}{(r-1)(\ell-1)-1}$
	(see~\thref{identitiesform2}).
	From this expression, we can see that 
	$\ell \mapsto \mm $ is decreasing.
	Thus, in order to conclude our proof, it suffices to show 
	that
	\[\dfrac{r+1}{2}- \dfrac{3}{2(r+3)}>m_{2}(K_{r},C_{4}).\]
	Using again the expression we have for $\mm$, an easy 
	calculation shows that the last inequality holds for every 
	$r \ge 4$. This completes the proof of the lemma.
\end{proof}

\section{The Algorithms}\label{sec:algos}

In this section, we formally describe the algorithm \hypertree\ 
and its subroutine \flower, and prove~\thref{hypertreeBasics}.
Let $n, r,\ell \ge 4$ be fixed integers throughout this section.

First, let us recall some notation from Section~\ref{proofofmain}.
Given any graph $G$, $\Crit_{r,\ell}(G)$ denotes the set of all 
$\star$-critical subhypergraphs of $\cG_{r,\ell}(G)$, 
the hypergraph whose hyperedges correspond to the copies of 
$K_{r}$ and $C_{\ell}$ on $G$.
For any $\mathcal{H}\subseteq \cG_{r,\ell}(G)$, we denote 
$\cE_1(\cH)=E(\cH) \cap \cE_{1}^{r}(\cG)$ and 
$\cE_2(\cH)=E(\cH) \cap \cE_{2}^{\ell}(\cG)$, where
\begin{align*}
	\cE_1^{r}(\cG) &= \{ E(F): F \cong K_r, F \subseteq G \} 
	\text{ and } 
	\cE_2^{\ell}(\cG) = 
	\{ E(F): F \cong C_{\ell}, F \subseteq G \}.
\end{align*}
The underlying graph of $\mathcal{H}$ is denoted by 
$\graph{\mathcal{H}}.$

We find it instructive to first provide an informal overview of 
\hypertree. 
This algorithm takes a hypergraph 
$\cH \in \Crit_{r,\ell}(G)$ as input, 
for some graph $G$ on $n$ vertices,
builds a sequence $(\cH_{i})_{i=0}^{T}$ of 
subhypergraphs of $\cH$ 
and outputs $\cH_{T}$.
The algorithm seeks to find a subhypergraph $\cF\se\cH$ for which 
the following holds.
The graph $F=\graph{\cF}$, which is a subgraph of $G$, 
satisfies (1) $\lambda(F)\le -\eps$
or (2) $\lambda (F) \le M$ and $e(F)\ge \log n$, 
for  some positive constants 
$\eps=\eps(r,\ell)$ and $M=M(r,\ell)$.

In the first step, the algorithm picks a hyperedge 
$E_{0} \in \cE_{1}(\cH)$ and sets $\cH_{0}=\{E_{0}\}$.
In the $i$-th step, it attaches a hyperedge 
$E_{i} \in \cE_{1}(\cH)$ 
to the current hypergraph $\cH_{i-1}$ to build $\cH_{i}$.
It is required that such a copy intersects $\graph{\cH_{i-1}}$ 
in at least two vertices, but is not a subgraph of 
$\graph{\cH_{i-1}}$. 
If no such hyperedge exists, 
then the algorithm runs a subroutine which we call \flower.
This algorithm, when called within \hypertree, returns 
(1) a hyperedge $C \in \cE_{2}(\cH)$
which intersects $\cH_{i-1}$ in at least one hypervertex and it 
is not contained in $\cH_{i-1}$; 
and (2) a collection of hyperedges in $\cE_{1}(\cH)$ so that each 
intersect $C$ in exactly one hypervertex.
The output of \flower\ is attached to $\cH_{i-1}$ to build 
$\cH_{i}$.
We defer the exact description of \flower\ until 
after the description of \hypertree.

The algorithm \hypertree\ also uses a 
\textit{canonical labelling function} to guarantee that the 
number of non-isomorphic underlying graphs given by output 
hypergraphs is not very large (cf.~\thref{lem:boundOut}).
Define 
$\sigma_{0}: E(K_{r})\to \big\{1,2,\ldots,\binom{r}{2}\big\}$ 
to be a fixed labelling of the edges of $K_{r}$.
For each sequence of underlying graphs
$(\graph{\cH_{0}},\ldots,\graph{\cH_{t}})$ generated by
\hypertree\ at step $t$,
we define its associated sequence of labellings 
$(\sigma_{i})_{i=0}^{t}$ in the following recursive way.
Given $\sigma_{i-1}: E(\graph{\cH_{i-1}})\to \mathbb{N}$,
define $\sigma_{i}: E(\graph{\cH_{i}})\to \mathbb{N}$ to be a 
function which satisfies the following properties:
(1) $\sigma_{i}(e)=\sigma_{i-1}(e)$ for every 
$e \in E(\graph{\cH_{i-1}})$;
(2) $\sigma_{i}$ labels the edges in
$E(\graph{\cH_{i}})\setminus E(\graph{\cH_{i-1}})$
with the next natural numbers larger than
$e(\graph{\cH_{i-1}})$; and
(3) outside $E(\graph{\cH_{i-1}})$, $\sigma_{i}$ 
is determined by the unlabelled graph generated by the edges in
$E(\graph{\cH_{i}})\setminus E(\graph{\cH_{i-1}})$
and by the set of vertices in $\graph{\cH_{i-1}}$ that
meet these edges.
We say that $\sigma_{i}$ is the \textit{canonical extension} of 
$\sigma_{i-1}$.
For each $i$, note that $\sigma_{i}$ is also an ordering of 
$V(\cH_{i})$.

A hyperedge $E$ always corresponds to a set of edges of some 
underlying graph, 
and so we denote by $V(E)$ the set $\cup_{e \in E} \ e$.
Next, is the formal description of the \hypertree\ algorithm.
Recall that $\eps=\eps(r,\ell)$ is the small positive constant given by \thref{dens-critic}.

\smallskip

\begin{algorithm}[H]
\DontPrintSemicolon
  
\KwInput{A hypergraph $\cH\in\Crit_{r,\ell}(G)$, 
for some graph $G$ on $n$ vertices}
	
\KwOutput{A pair $(\cH_T,D_T)$, where 
$\cH_T\se \cH$ and $D_T\subseteq \mathbb{N}$}

\tcc{Initialise:} 
  
\nl $i=1$, 
	$D_0 = \emptyset$, 
	$\cH_0 = \{E_0\}$ for some $E_0 \in \cE_1(\cH)$
	{\label{algH:initialise}}\;
	Let 
	$\sigma_{0}: E(V(\cH_{0}))\to \mathbb{N}$ be 
	equal to the canonical labelling of $K_{r}$\;

\nl \While{
			$\lambda(\graph{\cH_{0}}) > -\eps$ and 
			$i< \log n$
			}			
			{\label{algH:while}
           
			\nl \If{
						there exists $E\in \cE_1(\cH)$ 
						such that 
						$|V(E) \cap V(\graph{\cH_{i-1}})|\ge 2$ 
						and 
						$E\not\se V(\cH_{i-1})$
						\label{cliqueexistence}
					}							
					{\label{algHStep2If}
				 
						\nl set $\cH_{i} = \cH_{i-1}\cup\{E\}$
						and
						$D_{i} =D_{i-1}\cup\{i\}$ 
						{\label{algH:add1}}\; 
					}
			\Else{
					\nl let $\cH_F$ be the output of 
					\flower$(\cH_{i-1},\cH,\sigma_{i-1})$
					{\label{algH:flowerPowerBegin}}\;

					\nl set $\cH_{i} = \cH_{i-1} \cup \cH_F$ 
						{\label{algH:add2}}\;
							
					\nl \lIf{
								$|V(\graph{\cH_{i}})\sm 
								V(\graph{\cH_{i-1}})| = (r-1)
								(\ell-1)-1$ 
								{\label{algH:perfectflower}}
							}
							{
								set $D_{i}=D_{i-1}$
							}
					\nl	\lElse{set $D_{i}=D_{i-1}\cup\{i\}$}
					{\label{algH:flowerPowerEnd}}\;
				}
			\nl Let $\sigma_{i} : V(\cH_i)\to \mathbb{N}$ 
			be the canonical extension of $\sigma_{i-1}$ to 
			$\cH_i.${\label{algH:sigma}}\;
			\nl $i \mapsto i+1$ {\label{algH:increase}}\;}
   			\nl \Return{
						$(\cH_{i},D_{i})$
						}{\label{algH:return}}
							
\caption{\hypertree\label{HR}}
\end{algorithm}

\smallskip

Let us now turn to the subroutine \flower. 
The input of \flower\ is a triple $(\cH_i,\cH,\sigma_i)$, 
where $\cH$ is a $\star$-critical subhypergraph of 
$\cG_{r,\ell}(G)$,  for some graph $G$,
$\cH_i\se \cH$ and  $\sigma_i : V(\cH_i)\to \mathbb{N}$ is 
an ordering of the vertices of $\cH_i$.
When called within \hypertree,
the output is a subhypergraph of $\cH$ called \textit{flower}.
For a hyperedge $C$ of type $2$ and hyperedges 
$P_1,\ldots, P_t$ of type $1$,
we call the hypergraph $\cH_F=\{C,P_1,\ldots,P_t\}$
a \emph{flower} if $t<\ell$ and $|C \cap P_{s}|=1$ for all 
$1 \le s \le t.$
The hyperedges $P_{1},\ldots,P_{t}$ are called \emph{petals}.
Observe that, in the `graph world', a flower corresponds to a 
copy of $C_{\ell}$ and $t$ copies of $K_r$ that intersect $C$ in 
exactly one edge (and possibly more vertices).

We next introduce some non-standard notation that we need in the 
algorithm description to ensure that the total number of non-
isomorphic hypergraphs produced by \hypertree\ is not too big. 
For each $i \in \mathbb{N}$, 
let $\mathscr{C}(i,\cH_{i})$ be the set of all $\star$-critical
hypergraphs which generate the hypergraph $\cH_{i}$ in the $i$-th
step of \hypertree\ and which enter \flower\ at step $i+1$.
Now, define
\begin{align*}
	\overline{\cH}_{i} = \bigcup_{\cH \in \mathscr{C}(i,\cH_{i})} 
	\big\{E \in \cH: E \se V(\cH_{i}) \big\}.
\end{align*}
Observe that $\graph{\cH_{i}} = \graph{\overline{\cH}_{i}}$.
Now we are ready to state the formal description of \flower.

\smallskip

\begin{algorithm}[H]
\DontPrintSemicolon
  
\KwInput{A triple $(\cH_{i-1},\cH,\sigma_{i-1})$, 
	where $\cH$ is a hypergraph in $\Crit_{r,\ell}(G)$, 
	for some graph $G$, $\cH_{i-1}\se \cH$, and $
	\sigma_{i-1}$ is an ordering of $V(\cH_{i-1})$}
		
	\KwOutput{A flower $\cH_F=\{C\}\cup\{P_e : e\in C\sm 
	V(\cH_{i-1})\}$, 
	where $C \in \cE_{2}(\cH)$, 
	$P_{e}\in \cE_{1}(\cH)$ for all $e\in C\sm V(\cH_{i-1})$}
	\tcc{Find a seed:}
		 
	\nl Let $e_0 \in E(\graph{\cH_{i-1}})$ be the smallest 
	edge under the labelling $\sigma_{i-1}$ for which 
	$e_{0} \notin C$ for all 
	$C \in \cE_{2}(\overline{\cH}_{i-1})$\label{algF:seed}\;

	\nl let $C \in \cE_{2}(\cH)$ be a hyperedge containing 
	$e_{0}$ such that $C \not \se V(\cH_{i-1})$ 
	\label{algF:cycle}\;
		
 	\For{every $e\in C\sm V(\cH_{i-1})$}{
	\nl let $P_{e}\in \cE_1(\cH)$ be such that 
	$C\cap P_{e} = \{e\}$\label{algF:petalExistence}\;
 								   }
\Return{$\{C\}\cup\{P_e : e\in C\sm V(\cH_{i-1})\}$}
  
\caption{\flower\label{algF}}
\end{algorithm} 

\smallskip

{\setstretch{1.13}
In general, the algorithm \flower\ may return an error,
since an edge $e_{0}$ as in line~\ref{algF:seed} or a cycle $C$ 
as in line~\ref{algF:cycle} may not exist. 
However, such concerns are void when \flower\ is called within 
\hypertree, as the next lemma shows.

\begin{lemma}\thlabel{existenceofe0}
	Let $G$ be a graph on $n$ vertices and 
	$\cH\in\Crit_{r,\ell}(G)$. 
	Suppose that the algorithm \flower\ is called 
	in the $i$-th step of \hypertree$(\cH)$. 
	Then, the edge $e_{0}$ in line~\ref{algF:seed}, 
	the cycle $C$ in line~\ref{algF:cycle} and 
	the clique $P_e$ in line~\ref{algF:petalExistence} of 
	$\flower(\cH_{i-1},\cH,\sigma_{i-1})$ exist. 
	In particular, the algorithm runs without errors and 
	finishes in finite time.
	Moreover, $e_{0}$ is uniquely determined by $\cH_{i-1}$ and 
	$\sigma_{i-1}$.
\end{lemma}

\begin{proof}
	First, let us show the existence of $e_{0}$.
	We claim that $\overline{\cH}_{i-1}$ 
	cannot be $\star$-critical. 
	Otherwise, $\lambda(\graph{\cH_{i-1}})\le -\eps$,
	by~\thref{dens-critic},
	and this would imply that
	\hypertree$(\cH)$ has not entered the while loop in the 
	$i$-th step.
	In particular, \flower\ would not be called in the $i$-th
	step of \hypertree$(\cH)$.
	Now, observe that for every 
	$C \in \cE_{2}(\overline{\cH}_{i-1})$
	and every $e \in C$, there exists 
	$K \in \cE_{1}(\overline{\cH}_{i-1})$ 
	such that $K \cap C = \{e\}$.
	In fact, if $C \in \cE_{2}(\overline{\cH}_{i-1})$, 
	then there exists a hypergraph 
	$\cF \in \mathscr{C}(i-1,\cH_{i-1})$
	such that $C \in \cE_{2}(\cF)$. 
	As $\cF$ is $\star$-critical, there must also exist 
	$K \in \cE_{1}(\cF)$ such that $K \cap C = \{e\}$.
	But, because \hypertree$(\cF)$ has entered the 
	else-statement at step $i$, 
	we have $K \se V(\cH_{i-1})$, and hence
	$K \in \cE_{1}(\overline{\cH}_{i-1})$.
	It follows that the only reason for which 
	$\overline{\cH}_{i-1}$
	is not $\star$-critical is because there exists an edge 
	$e \in E(\graph{\cH_{i-1}})$
	for which there is no hyperedge in 
	$\cE_{2}(\overline{\cH}_{i-1})$ containing $e$.
	Among these edges, we take $e_{0}$ to be the one with 
	the smallest labelling under $\sigma_{i-1}$.
	By construction, $e_{0}$ only depends on $\cH_{i-1}$ and 
	$\sigma_{i-1}$.
	As $\cH$ is $\star$-critical, 
	the existence of the cycle $C$ as in line~\ref{algF:cycle} 
	and the petals $P_{e}$ as in line~\ref{algF:petalExistence} 
	is straightforward.
\end{proof}
}

\begin{lemma}\thlabel{lem:preludeflowerCorrectness}
	Under the same assumptions as in~\thref{existenceofe0},
	$\flower(\cH_{i-1},\cH,\sigma_{i-1})$ output 
	a flower $\cH_F=\{C\}\cup\{P_e : e\in C\sm V(\cH_{i-1})\}$ 
	which satisfies the following properties:
	\begin{enumerate}
	\item[\mylabel{CnotincH_i}{$(\textit{F1})$}] 
	$C\in \cE_2(\cH)$ but $C\not\se V(\cH_{i-1}),$
		
	\item[\mylabel{petalsandcycleintersection}{$(\textit{F2})$}] 
	$P_e\in \cE_1(\cH)$ and $C\cap P_e =\{e\}$ 
	for every $e\in C\sm V(\cH_{i-1}),$ and
		
	\item[\mylabel{intersectionwithG_i}{$(\textit{F3})$}]
	$|V(P_{e})\cap V(\graph{\cH_{i-1}})| \le 1 \le 
	|C \cap E(\graph{\cH_{i-1}})|$ 
	for all $e\in C\sm V(\cH_{i-1}).$
	\end{enumerate}
\end{lemma}

\begin{proof}
	The properties~\ref{CnotincH_i}
	and~\ref{petalsandcycleintersection} 
	are immediate from the algorithm description and 
	Lemma~\ref{existenceofe0}. 
	As \flower\ was called 
	in the $i$-th step of \hypertree$(\cH)$, 
	line~\ref{cliqueexistence} of \hypertree\ 
	was not executed.
	This means that for each petal $P_{e}$ we have 
	$|V(P_{e})\cap V(\graph{\cH_{i-1}})| \le 1$,
	which proves the first inequality 
	in~\ref{intersectionwithG_i}.
	The second inequality follows from the existence of $e_{0}$ 
	in line~\ref{algF:seed}, as 
	$e_{0} \in C \cap E(\graph{\cH_{i-1}})$.
\end{proof}

With Lemmas~\ref{existenceofe0} and~\ref{lem:preludeflowerCorrectness} at our hands, we now deduce \thref{hypertreeBasics}.

\begin{proof}[Proof of \thref{hypertreeBasics}]
	%\ref{hypertreeBasics-E0}, \ref{hypertreeBasics-stopping} and 
	%\ref{hypertreeBasics-return}] 
	In its initialisation, \hypertree$(\cH)$ algorithm sets
	$D_0 = \emptyset$ and $\cH_0 = \{E_0\}$, for some 
	$E_0 \in \cE_1(\cH)$. 
	This already establishes Part \ref{hypertreeBasics-E0}.
	Now, for each step $i$ of the while loop, 
	\hypertree$(\cH)$ executes one of the following actions.
	Either it sets 
	$\cH_{i}=\cH_{i-1}\cup \{E\}$ for some $E \in \cE_{1}(\cH)$ 
	(Case 1, cf.~line~\ref{algH:add1}), 
	or it sets $\cH_{i}=\cH_{i-1}\cup \cH_{F}$, where
	$\cH_{F}$ is the output of the algorithm 
	\flower$(\cH_{i-1},\cH,\sigma_{i-1})$ 
	(Case 2, cf.~line~\ref{algH:add2}). 
	We remark that one of them must be executed because
	$\cH_{i-1}\notin \Crit_{r,\ell}(G)$,
	as $\lambda(\graph{\cH_{i-1}})>-\eps$ 
	(see~\thref{dens-critic}).
	In either case, $\cH_{i-1} \subseteq \cH_{i} \subseteq \cH$
	(cf.~\thref{lem:preludeflowerCorrectness} 
	for the second case).
	Similarly, it is easy to see from 
	lines~\ref{algH:add1},~\ref{algH:perfectflower} 
	and~\ref{algH:flowerPowerEnd} that 
	$D_{i-1} \subseteq D_{i} \subseteq \mathbb{N}$.

	Let $T$ be the number of iterations of the while loop in
	line~\ref{algH:while} of \hypertree$(\cH)$.  
	By the while loop condition and 
	the increase of $i$ by one in every iteration 
	(see line~\ref{algH:increase}), 
	\hypertree$(\cH)$ must stop in at most $\log n$ iterations.
	Moreover, the while loop guarantees that $T$ is the smallest 
	integer such that $\lambda(G_T)\le -\eps$ or $T \ge \log n$, 
	where $G_T=\graph{\mathcal{H}_T}$ and $\eps$ 
	is the constant given by~\thref{dens-critic}.
	This establishes Part~\ref{hypertreeBasics-stopping}.
	Since \hypertree$(\cH)$ should return $(\cH_{T},D_{T})$ 
	(see line~\ref{algH:return}), we also establish
	Part~\ref{hypertreeBasics-return}.

	Now, it remains to show Part~\ref{hypertreeBasics-vertices}.
	Let $i$ be any step of the while loop in \hypertree$(\cH)$.
	In Case 1, the hyperedge $E$ satisfies 
	$E\not\se V(\cH_{i-1})$ (cf.~line~\ref{algH:add1}).
	In Case 2,~\thref{lem:preludeflowerCorrectness} implies that
	the output $\cH_{F}=\{C,P_{1},\ldots,P_{t}\}$ given by
	\flower$(\cH_{i-1},\cH,\sigma_{i-1})$ satisfies 
	$C \not \se V(\cH_{i-1})$.
	In both cases, $v(\cH_{i-1})<v(\cH_{i})$.
\end{proof}

\section{The algorithm analysis}\label{sec:analysis}
In this section, we prove \thref{deg-lambda,lem:boundOut},
and hence complete the proof of~\thref{thm:maintech}. 
The former states that in each step of the algorithm, $\lambda(G_i)$ either decreases by an additive constant in a degenerate step, or that its value remains the same. The latter then states that the number of non-isomorphic structures that the algorithm can generate is not too big.

Recall the description of \hypertree.
Throughout this section, let $G_{i}$ denote the graph 
$\graph{\cH_{i}}$, where $\cH_{i}$ is the hypergraph generated in $i$-th step of \hypertree.
For all $1 \le i \le T$, the graph $G_i$ is obtained from $G_{i-1}$ by adding either an $r$-clique or an underlying graph of a flower $\{C,P_1,\ldots,P_t\}$ to it, depending whether \hypertree\ executes the if-clause in lines 
\ref{algHStep2If}--\ref{algH:add1} or the else-clause in lines \ref{algH:flowerPowerBegin}--\ref{algH:flowerPowerEnd}. In the latter case, we will analyse the change in $\lambda$ by adding first $C$, and then one petal (copy of $K_r$) at a time.
Thus, it makes sense to pin down the effect of adding a copy of $K_r$ to an arbitrary graph $F$ first. 

For two graphs $F_1$ and $F_2$, we denote by $F_1\cap F_2$ the subgraph with vertex set $V(F_1)\cap V(F_2)$ and edge set $E(F_1)\cap E(F_2).$ The graph $F_1\cup F_2$ is defined analogously. 
For any graph $F$, recall that 
$\lambda(F) = v(F)-e(F)/m_2(K_r,C_{\ell})$.
Then, we can write 
\begin{align}\label{eq:lambdaDiff-0}
\lambda(F_1\cup F_2) - \lambda (F_1) &= v(F_1\cup F_2)-v(F_1) - \frac{e(F_1\cup F_2)-e(F_1)}{m_2(K_r,C_{\ell})} \nonumber \\
&= v(F_2)-v(F_1\cap F_2) - \frac{e(F_2)-e(F_1\cap F_2)}{m_2(K_r,C_{\ell})}.
\end{align}
Now, define
\begin{align}\label{def:beta}
\beta_{r,\ell}(J)&= r-v(J) - \frac{\binom{r}{2}-e(J)}{m_2(K_r,C_{\ell})}.
\end{align} 
By~\eqref{eq:lambdaDiff-0}, we have
\begin{align}\label{eq:lambdaDiff}
\lambda(F_1\cup F_2) - \lambda (F_1) &= \beta_{r,\ell}(J)
\end{align}
in the case when $F_2 \cong K_r$ and $J = F_1\cap F_2.$ 
Before stating the lemma which encompasses how 
$\beta_{r,\ell}(J)$ behaves for various subgraphs $J\se K_r$,
we need the following lemma which provides closed formulas for 
$m_{2}(C_{\ell})$, $m_{2}(K_{r})$ 
and $m_{2}(K_{r}, C_{\ell})$.
We prove it in the appendix. 

\begin{lemma}\thlabel{identitiesform2} 
	Let $r,\ell \ge 4$ be integers. Then,
	\begin{align*}
		m_2(C_{\ell}) = \dfrac{\ell-1}{\ell-2}, \quad 
		m_{2}(K_{r}) = \dfrac{r+1}{2}, \quad
		\text{and} \quad \mm = 
		\dfrac{\binom{r}{2}}{r-2+(\ell-2)/(\ell-1)}.
	\end{align*}
	In particular, $r/2 < \mm < m_{2}(K_{r})$.
\end{lemma}

In our next lemma, 
we obtain upper bounds for $\beta_{r,\ell}(J)$ 
for every subgraph $J\subsetneq K_r$ with at least two vertices.

\begin{lemma}\thlabel{claim:beta}
	Let $r,\ell \ge 4$ be integers.
	Let $J\subsetneq K_r$ such that $v(J)\ge 2$. Then,
	\begin{enumerate}
	\item[\mylabel{beta:less0}{$(a)$}]
	$\beta_{r,\ell}(J) < 0$,
 
	\item[\mylabel{beta:K2}{$(b)$}]
	$\beta_{r,\ell}(K_2) = 
	1/m_2(K_r,C_{\ell})- (\ell-2)/(\ell-1)> -1$,

	\item[\mylabel{beta:lessK2}{$(c)$}]
	$\beta_{r,\ell}(J) \le \beta_{r,\ell} (K_2)$ if 
	$d(v) = 1$ for some $v \in V(J)$. 
	The equality holds if and only if $J\cong K_2$.
\end{enumerate}
\end{lemma}

\begin{proof}

	First, let us prove part~\ref{beta:less0}.
	When $J \subsetneq K_{r}$ has $r$ vertices,
	we can easily see from~\eqref{def:beta} that
	$\beta_{r,\ell}(J) < 0$.
	Thus, let us assume that $2 \le v(J) < r$.
	Observe that $\beta_{r,\ell}(J) < 0$ if the following 
	inequalities are satisfied:
	\begin{align}\label{mminequalitybk2}
		\mm < m_{2}(K_{r}) \le \dfrac{\binom{r}{2}-e(J)}{r-v(J)}.
	\end{align}
	It remains to show that both inequalities 
	in~\eqref{mminequalitybk2} are true.
	The first inequality follows from~\thref{identitiesform2}.
	As $m_{2}(K_{r}) = \big ( \binom{r}{2}-1 \big )/(r-2)$
	by~\thref{identitiesform2},
	the last inequality in~\eqref{mminequalitybk2} is equivalent 
	to
	\begin{align}\label{inequalitym_2(K_r)}
		\dfrac{\binom{r}{2}-1}{r-2} \le 
		\dfrac{\binom{r}{2}-e(J)}{r-v(J)}.
	\end{align}
	If $v(J) = 2$, then $e(J) \le 1$ and 
	hence~\eqref{inequalitym_2(K_r)} 
	holds.
	If $3\le v(J)< r$, then the last inequality can be rearranged 
	to
	\begin{align*}
		\dfrac{e(J)-1}{v(J)-2} \le \dfrac{\binom{r}{2}-1}{r-2}.
	\end{align*}
	As $e(J) \le \binom{v(J)}{2}$ and 
	$\big ( \binom{r}{2}-1\big )/(r-2) = (r+1)/2$,  
	the last inequality holds whenever $v(J) \le r$.
	This establishes part~\ref{beta:less0}.

	To show part~\ref{beta:lessK2}, first note that 
	$d(v)=1$ for some $v \in V(J)$ if and only if 
	$K_2\se J \se K_{r-1}\cdot K_2$,
	where $K_{r-1}\cdot K_2$ denotes the graph obtained from 
	$K_{r-1}$ by adding a pendant edge. 
	When $J \cong K_2$, the equality in~\ref{beta:lessK2} holds 
	trivially. 
	Thus, let us assume that 
	$v(J)\ge 3$ and $J \se K_{r-1}\cdot K_2$.
	In this case,
	the inequality $\beta_{r,\ell}(J)< \beta_{r,\ell}(K_2)$ is 
	equivalent to 
	\begin{align}\label{eq:aux:J1}
		\dfrac{e(J)-1}{v(J)-2} < \mm.
	\end{align}
	But, for any $J \se K_{r-1}\cdot K_2$ such that $v(J)\ge 3$,
	we have  
	\begin{align*}
		\dfrac{e(J)-1}{v(J)-2} \le m_2(K_{r-1}\cdot K_2) 
		= \max \left \{ m_2(K_{r-1}), 
		\dfrac{e(K_{r-1}\cdot K_2)-1}{
		v(K_{r-1}\cdot K_2)-2}\right\} 
		= \dfrac{r}{2},
	\end{align*}
	by definition of $m_2(\cdot)$ and the identity 
	$m_{2}(K_{r-1})=r/2$ (see ~\thref{identitiesform2}). 
	As $m_2(K_r,C_{\ell}) > r/2$ by~\thref{identitiesform2},
	this finishes the proof of part~\ref{beta:lessK2}.

	For part~\ref{beta:K2}, the identity 
	$\beta_{r,\ell}(K_2) = 1/m_2(K_r,C_{\ell})- 
	(\ell-2)/(\ell-1)$ 
	follows readily from the definition of $\beta_{r,\ell}$ 
	in~\eqref{def:beta} and the identity for $\mm$ 
	in~\thref{identitiesform2}. 
	Finally, $m_2(K_r,C_\ell) >0$ and $(\ell-2)/(\ell - 1) <1$ 
	imply that $\beta_{r,\ell}(K_2) > -1$.

\end{proof}

Now we are ready to prove~\thref{deg-lambda}.

\begin{proof}[Proof of \thref{deg-lambda}] 

	Suppose that \hypertree(\cH) executes the if-statement in 
	lines \ref{algHStep2If}--\ref{algH:add1} in the $i$-th 
	iteration of its while loop. 
	Then, $i \in D_{i}$ and hence $i \in D_{T}$,
	by~\thref{hypertreeBasics}.
	Moreover, $G_{i} = G_{i-1} \cup K$ for some $K \cong K_{r}$ 
	such that $|V(G_{i-1})\cap V(K)|\ge 2$ and 
	$K\not\se G_{i-1}$.
	Observe that graph $J=G_{i-1}\cap K$ satisfies the 
	assumptions of \thref{claim:beta} and hence, 
	by~\eqref{eq:lambdaDiff},  
	$\lambda(G_i)-\lambda(G_{i-1})=\beta_{r,\ell}(J) < 0.$ 
	%Take $\delta_{1} = - \max \beta_{r,\ell}(J)$,
	%where the maximum is taken over all $J\subsetneq K_r$ 
	%such that $v(J)\ge 2$. Then, 
	%$\lambda(G_i)\le\lambda(G_{i-1})-\delta_{1}$.

	Now, suppose that \hypertree(\cH) executes the else-statement 
	in lines 
	\ref{algH:flowerPowerBegin}--\ref{algH:flowerPowerEnd} in 
	the $i$-th iteration of its while loop. 
	Let $\cH_F = \{C\}\cup \{P_e : e\in C\sm E(G_{i-1})\}$ be the 
	flower returned by \flower$(\cH_{i-1},\cH,\sigma_{i-1})$. 
	Recall all the properties of $\cH_{F}$ given by
	~\thref{lem:preludeflowerCorrectness}.
	In order to bound the difference 
	$\lambda(G_{i})-\lambda(G_{i-1})$, 
	we first analyse the increment 
	$\lambda (G_{i-1}\cup C)-\lambda(G_{i-1}).$ 
	Let $J_0$ be the graph $G_{i-1}\cap C$.
	By~\eqref{eq:lambdaDiff-0}, we have
	\begin{align} \label{dlambdaC}
		\lambda (G_{i-1}\cup C)-\lambda(G_{i-1}) &=
		\ell-v(J_0)-\dfrac{\ell-e(J_0)}{m_2(K_r,C_{\ell})} 
		\nonumber\\
		& \le \left(\dfrac{\ell-2}{\ell-1}-
		\dfrac{1}{m_2(K_r,C_{\ell})}\right) \cdot 
		|E(C)\sm E(G_{i-1})|,
	\end{align}
	where in the inequality we use that $v(J_0)\ge 2$ and 
	$v(J_0)> e(J_0)$, as $K_{2}\se J_0\subsetneq C_{\ell}$
	(see~\thref{lem:preludeflowerCorrectness}).
	Note that equality holds in~\eqref{dlambdaC} if and only if 
	$J_0\cong K_2.$

	By~\thref{claim:beta}\ref{beta:less0}, 
	the contribution of each
	petal of $\cH_{F}$ to $\lambda$ is negative.
	But, the contribution of $C$ to $\lambda$, 
	which is bounded by~\eqref{dlambdaC}, may be positive
	(and large). 
	However, as we shall show, 
	the contribution of $C$ to $\lambda$ is smaller or equal 
	than the absolute value of the sum of all the contributions 
	of each petal of $\cH_F$ to $\lambda$.
	In order to prove this,
	we recursively find a subsequence of petals 
	$(P_{j})_{j=1}^{t}$ in $\cH_{F}$ such that 
	the intersection graph 
	$P_{j}\cap (G_{i-1}\cup C \cup P_{1}\cup 
	\cdots \cup P_{j-1})$ has potentially many isolated 
	vertices.
	These isolated vertices allow us to gain a sufficiently 
	negative contribution to $\lambda$ from each petal in the 
	sequence, 
	and hence `beat' the contribution given by the cycle 
	in~\eqref{dlambdaC}.
	This sequence of petals does not necessarily contain all the 
	petals of $\cH_{F}$, but this is not a problem.
	By \thref{claim:beta}\ref{beta:less0},
	all the petals in $\cH_{F}$ give a negative contribution to
	$\lambda$, and hence we may discard some petals from the 
	analysis 
	(and adding them later will not increase the value of 
	$\lambda$).

	We define this sequence of petals iteratively. 
	Let $\{u_0,\ldots u_{\ell -1}\}$ be a cyclic ordering of the 
	vertices of $C$ such that $u_0u_{\ell-1}$ is an edge of 
	$G_{i-1}\cap C,$ which must exist by
	property \ref{intersectionwithG_i} of
	\thref{lem:preludeflowerCorrectness}.  
	By convention, define $u_{-1}=u_{\ell-1}$.
	Now, define $A_0 = E(C)\sm E(G_{i-1})$ 
	and construct a nested sequence of sets $(A_{s})_{s \ge 0}$
	in the following recursive way.
	For each $s \in \mathbb{N}$, let
	\begin{align*}
		m_s = \min\{m : u_{m}u_{m+1} \in A_{s-1}\}
	\end{align*}
	and let $P_s = P_{u_{m_{s}}u_{m_{s}+1}}$ be the unique petal
	in $\cH_{F}$ which covers the edge $u_{m_s}u_{m_s+1}$,
	meaning that  
	$P_s\cap C = \{u_{m_s}u_{m_s+1}\}$.
	Then, set 
	\begin{align*}
		A_s = A_{s-1} \sm \{u_mu_{m+1} : u_{m+1} \in V(P_s) \}.
	\end{align*}
	Let $t$ be the smallest integer such that $A_{t}=\emptyset$,
	and note that $t \le |A_{0}| = |E(C) \sm E(G_{i-1})|$.

	For simplicity, denote $G_{i-1}^{(0)} =  G_{i-1}\cup C$ and, 
	more generally, for each $1\le s \le t$, let 
	\begin{align*}
		G_{i-1}^{(s)} =  
		G_{i-1}\cup C \cup P_1 \cup\ldots \cup P_s.
	\end{align*} 
	Now, observe that 
	\begin{align}\label{conclusion1.0}
		\lambda (G_{i})-\lambda(G_{i-1}) 
		\le \lambda (G_{i-1}^{(0)})-\lambda(G_{i-1}) 
		+\sum_{s=1}^t \lambda (G_{i-1}^{(s)})
		-\lambda(G_{i-1}^{(s-1)}).
	\end{align}
	Indeed, \thref{claim:beta}\ref{beta:less0} together 
	with~\eqref{eq:lambdaDiff} imply that
	we can discard the petals in 
	$\{P_{e}: e\in E(C)\sm E(G_{i-1})\}$ which do not belong to
	the chosen sequence $P_1,\ldots,P_t$.
	Moreover, equality holds if and only if 
	$P_1,\ldots,P_t$ are all the petals in the flower
	$\cH_{F}$.
	To bound each increment in~\eqref{conclusion1.0},
	we next analyse the structure of the graph 
	$J_{s}:= G_{i-1}^{(s-1)}\cap P_s$.
	Define
	\begin{align*}
	I_s = \{u_{m+1} \in V(J_s)\sm\{u_{m_s+1}\}: 
	u_mu_{m+1} \in A_{s-1}\}.
	\end{align*}

	\begin{claim}\thlabel{claim:deg1}
		The degree of $u_{m_s+1}$ in $J_s$ is 1 and $I_s$ 
		is a set of isolated vertices in $J_s$.
	\end{claim}

	\begin{proof}
		Let $u_{m}$ be any vertex in $I_{s}\cup \{u_{m_{s}+1}\}$
		and $w$ be any vertex in $J_{s}$.
		We affirm that $u_{m}$ is adjacent to
		$w$ inside the graph $J_{s}$ 
		if and only if $\{u_{m}w\} = C\cap P_s$. 
		Indeed, we cannot have
		$\{u_{m}w\}\in E(G_{i-1})$,
		otherwise $P_{s}$ would be an $r$-clique which intersects
		$G_{i-1}$ in at least $2$ vertices,
		contradicting~\thref{lem:preludeflowerCorrectness}
		\ref{intersectionwithG_i}.
		If $w \neq u_{m+1}$, we also cannot have 
		$\{u_{m}w\} \in E(P_{1}\cup \cdots \cup P_{s-1})$,
		otherwise $u_{m-1}u_{m} \notin A_{s-1}$, and hence
		$u_{m}\notin I_{s} \cup \{ u_{m_{s}+1}\}$.
		As $u_{m_s}u_{m_s+1}$ is the only edge in $P_s\cap C$,
		it follows that $u_{m_{s}}$ is the only neighbour of 
		$u_{m_{s}+1}$ in $J_s$, and that 
		$u_{m}$ is isolated in $J_{s}$ for any $u_{m}\in I_{s}$.
	\end{proof}

	Let $\wt J_s$ be the subgraph of $J_s$ 
	induced by the vertex set $V(J_s)\sm I_s$.
	By the previous claim, we have $E(\wt J_s) = E(J_{s})$, 
	which implies that 
	$\beta_{r,\ell}(J_s) = \beta_{r,\ell}(\wt J_s)-|I_s|$ 
	(see~\eqref{def:beta}).
	By~\eqref{eq:lambdaDiff}, we obtain 
	\begin{align}\label{change1petal}
		\lambda (G_{i-1}^{(s)})-\lambda(G_{i-1}^{(s-1)}) 
		&=\beta_{r,\ell}(J_s) 
		=\beta_{r,\ell}(\wt J_s)-|I_s|
		\le \beta_{r,\ell}(K_2) -|I_s|
	\end{align}
	for every $1\le s\le t$.
	In the last inequality
	we use \thref{claim:beta}\ref{beta:lessK2},
	as $d(u_{m_{s}+1})=1$ by \thref{claim:deg1}.
	Moreover, by \thref{claim:beta}\ref{beta:lessK2}, 
	equality holds if and only if $\wt J_s \cong K_2$.
	When $\wt J_s \cong K_2$, note that we also have
	$|I_{s}|=0$, as the only vertex $u_{m+1} \in V(J_{s})$ such 
	that $u_{m}u_{m+1}\in A_{s-1}$ is $u_{m+1}=u_{m_{s}+1}$.

	Combining~\eqref{dlambdaC},~\eqref{conclusion1.0} 
	and~\eqref{change1petal}, we have
	\begin{align}\label{conclusion1.1}
		\lambda (G_{i})-\lambda(G_{i-1}) 
		&\le \left( \dfrac{\ell-2}{\ell-1}-
		\dfrac{1}{m_2(K_r,C_{\ell})} \right) 
		\cdot |E(C)\sm E(G_{i-1})|
		+t\beta_{r,\ell}(K_2)
		-\sum_{s=1}^t |I_s|.
	\end{align}
	From the definitions of $A_{s}$ and $I_{s}$, 
	it is easy to see that
	$\sum_s (|I_s|+1) =|A_0|=|E(C)\sm E(G_{i-1})|$.
	And, by Lemma~\ref{claim:beta}\ref{beta:K2},
	we have 
	$\beta_{r,\ell}(K_2) = 
	m_2(K_r,C_{\ell})^{-1} - (\ell-2)/(\ell-1)$.
	Then,~\eqref{conclusion1.1} is equivalent to
	\begin{align}\label{conclusion2}
		\lambda (G_{i})-\lambda(G_{i-1}) \le
		(\beta_{r,\ell}(K_{2})+1)\cdot (t-|A_{0}|).
	\end{align}
	By \thref{claim:beta}, $\beta_{r,\ell}(K_2) > -1$ and,
	as we have $t \le |A_{0}|$, it follows that 
	\begin{align}\label{conclusion3}
		(\beta_{r,\ell}(K_{2})+1)\cdot (t-|A_{0}|)\le 0.
	\end{align}
	Clearly, equality in~\eqref{conclusion3} holds if and only 
	if $t = |A_{0}|$.
	We conclude that $\lambda (G_{i})-\lambda(G_{i-1}) \le 0$ in
	the case when we add the flower 
	$\{C\}\cup\{ P_e: e\in  C\sm E(G_{i-1})\}$.

	Observe that $\lambda (G_{i})-\lambda(G_{i-1}) = 0$
	if and only if we have equalities in
	~\eqref{dlambdaC}--\eqref{conclusion3}.
	This means that we must have $C \cap G_{i-1} \cong K_{2}$
	(and hence $|A_{0}|=\ell-1$),
	$t = |A_{0}|$ and 
	\begin{align*}
		 P_s\cap (G_{i-1}\cup C\cup P_1 \cup \cdots \cup P_{s-1}) 
		 \cong K_2,
	\end{align*}
	for each $1 \le s \le t$.
	As $e \in E(P_{e}\cap C)$, we infer that
	none of the $\ell-1$ petals intersect outside the cycle $C$
	and that the only petals sharing a vertex are consecutive 
	petals,
	which share exactly one vertex.
	This happens if and only if 
	$|V(G_i)\sm V(G_{i-1})| = (r-1)(\ell -1) -1$,
	whence $i$ is {\em not} added to $D_i$
	(cf. line~\ref{algH:perfectflower} of \hypertree), and then 
	$i \notin D_T$. 
	This proves~\ref{item:1@deg-lambda}. 
	The existence of $\delta = \delta(r,\ell)$ 
	for~\ref{item:2@deg-lambda} readily follows by noting that 
	there are only $C=C(r,\ell)$ non-isomorphic configurations 
	of such flowers and cliques (and how they intersect with 
	$G_{i-1}$). 
	This finishes the proof of the lemma.
\end{proof}

It remains to prove \thref{lem:boundOut}, 
which bounds the number of non-isomorphic underlying graphs of
hypergraphs that \hypertree\ may output. 
In principle, $|\Out_{r,\ell}(n)|$ could be very large, 
but this is avoided with the help of 
the canonical labelling function 
(recall its definition given before the description of 
\hypertree).
For each $t=1,\ldots,\lceil\log n \rceil$,
these vertex labellings assist the construction of 
$\graph{\cH_{t}}$ from $\graph{\cH_{t-1}}$ and,
for all but a constant number of steps,
we will see that this construction is unique.
That is, it does not depend on the input, it depends only on 
$\graph{\cH_{t-1}}$ and $\sigma_{t-1}$.
This is the reason why we get at most a polylogarithmic bound on 
the number of outputs.

In order to bound $|\Out_{r,\ell}(n)|$,
we first bound how many pairs $(\graph{\cH_t},\sigma_t)$ 
\hypertree\ can produce in step $t$,
for all $t=1,\ldots,\lceil\log n \rceil$.
To do so, we need to recall and define some notation.
For a $\star$-critical hypergraph $\cH$, 
let $T(\cH)$ be the stopping time of \hypertree$(\cH)$.
For any $t\ge 0$ and any $\star$-critical hypergraph $\cH$ 
such that $T(\cH)\ge t$, let $\cH_{t}(\cH)$ be the hypergraph 
obtained in step $t$ of \hypertree$(\cH)$.
Let $D_t(\cH)$ be the accompanying set and 
$\sigma_t(\cH): V(\cH_{t})\to \mathbb{N}$ be the
canonical vertex labelling of $\cH_{t}$, 
and hence of $E(\graph{\cH_{t}})$. 
For each $n \in \mathbb{N}$, define $\Crit_{r,\ell}(n) = \bigcup_{v(G)=n} \Crit_{r,\ell}(G)$.
Finally, for each $t,n \in \mathbb{N}$ and each set 
$D\subseteq \{1,\ldots,t\}$, define 
\begin{align*}
	\cG(t,D,n) = \bigcup \left\{ (\graph{\cH_t},\sigma_t) : 
	\cH_t = \cH_t(\cH) ,\sigma_t = \sigma_t(\cH)\right\},
\end{align*} 
where the union is over all 
$\cH \in \Crit_{r,\ell}(n)$ such that $D_{t}(\cH)=D$  and  $T(\cH) \ge t$.

Our next lemma gives an upper bound on the size of 
$\mathcal{G}(t,D,n)$.

{\setstretch{1.12}
\begin{lemma}\thlabel{Fdk}
	For all $r,\ell \ge 4$ there exists $C >0$ such that 
	$|\cG(t,D,n)| \leq t^{C|D|}$, for all 
	$ t,n \in \mathbb{N}$ and $D \in \{1,\ldots,t\}$. 
\end{lemma}

\begin{proof}
	To simplify notation, set $g(t,D,n):=|\mathcal{G}(t,D,n)|$.
	First, note that $\cG(0,\emptyset,n)$ contains only one pair, 
	and hence $g(0, \emptyset,n)=1$.
	In fact, for every $\star$-critical hypergraph $\cH$, 
	$\cH_0(\cH)$ consists of a single hyperedge of type $1$ 
	(cf.~\thref{hypertreeBasics}\ref{hypertreeBasics-E0}) and 
	$\sigma_0$ is a fixed labelling of $V(\cH_0)$.

	Now, we claim that for each $t\ge 1$ and each 
	$D \subseteq \{1,\ldots,t\}$, we have
	\begin{align}\label{boundg(t,D)}
		g(t,D,n)\le 
		\begin{cases} 
		g(t-1,D,n) &\text{if } t \notin D; \\
		g\big( t-1,D\setminus \{t\},n \big)\cdot 
		(4t\ell r^2)^{\ell r^2} &\text{if } t \in D.
		\end{cases}
	\end{align}

	First, assume that $t\not\in D$ and let 
	$(\graph{\cH_{t}},\sigma_{t})\in \cG(t,D,n).$ 
	Let $\cH$ be any hypergraph in 
	$\Crit_{r,\ell}(n)$ 
	such that $\cH_{t}(\cH)= \cH_{t}$, 
	$\sigma_{t}(\cH)=\sigma_{t}$ and
	$D_{t}(\cH)=D$.
	Note that $t\not\in D$ implies that 
	$D_{t}(\cH)=D_{t-1}(\cH)$.
	And this happens if and only if 
	$\cH_{t}=\cH_{t-1}\cup \cH_{F}$,
	for some flower $\cH_{F}$ such that
	\begin{align}\label{perfectflower}
		|V(\graph{\cH_{t}})\sm V(\graph{\cH_{t-1}})|=
		(r-1)(\ell-1)-1,
	\end{align}
	see~line~\ref{algH:perfectflower} of \hypertree.  
	For~\eqref{perfectflower} to hold, observe that
	$G(\cH_{F})$ must intersect $\graph{\cH_{t-1}}$ in exactly
	one edge, which is given by line~\ref{algF:seed} of 
	$\flower(\cH_{t-1},\cH,\sigma_{t-1})$. 
	Once we have this edge, called $e_{0}$, we can see 
	that \flower\ generates only one type of flower $\cH_{F}$ 
	such that~\eqref{perfectflower} holds and 
	$\graph{\cH_{F}}\cap \graph{\cH_{t-1}}$ 
	is equal to $e_{0}$.
	Moreover, by \thref{existenceofe0}, 
	$e_0$ only depends on 
	$\graph{\cH_{t-1}}$ and the canonical labelling 
	$\sigma_{t-1}$.
	Therefore, for any other hypergraph $\cH'$
	such that $\graph{\cH_{t-1}(\cH')} = \graph{\cH_{t-1}(\cH)}$,
	$\sigma_{t-1}(\cH')=\sigma_{t-1}$
	and $t \notin D_{t}(\cH')$,
	the flower $\cH_{F}'$ given by 
	$\flower(\cH_{t-1},\cH',\sigma_{t-1})$
	also satisfies $V(\cH') \cap V(\cH_{t-1})=\{e_{0}\}$
	for the same $e_{0}$.
	As $\sigma_t$ is the canonical labelling 
	extending $\sigma_{t-1}$ to $\cH_t$, 
	which is uniquely determined by 
	$\graph{\cH_{t-1}}$
	and the unlabelled graph $\graph{\cH_{F}}$,
	we also have $\sigma_{t}(\cH') = \sigma_{t}$.
	This implies that there is an injection 
	$\cG(t,D,n) \to \cG(t-1,D,n)$ mapping 
	$(\cH_{t},\sigma_{t})$ to $(\cH_{t-1},\sigma_{t-1})$,
	and hence $g(t,D,n)\le g(t-1,D,n)$. 
	This proves the first inequality in~\eqref{boundg(t,D)}.

	To show the second inequality, note that in step $t$ of 
	\hypertree(\cH) one of the following holds:
	(1) The algorithm has stopped; 
	(2) $\cH_{t}= \cH_{t-1}\cup \{E\}$, 
	for some $E \in \cE_{1}(\cH)$; or
	(3) $\cH_{t}= \cH_{t-1}\cup \{\cH_{F}\}$, 
	for some flower $\cH_{F} \se \cH$.
	Let $H = \graph{E}$ or $H = \graph{\cH_{F}}$ be the
	underlying graph of the hyperedges that are added.
	Note that $v(H) \le \ell r^{2}$ and,
	as the number of vertices in $\graph{\cH_{t}}$ is at most 
	$t\ell r^{2}$, there are at most 
	$(t\ell r^{2})^{\ell r^{2}}\cdot 2^{\ell r^{2}}$ ways to 
	choose the subgraph $H \cap \graph{\cH_{t-1}}$.
	Once this subgraph is fixed, there are at most 
	$2^{\ell r^{2}}$
	ways to choose the edges of $H \setminus \graph{\cH_{t-1}}$
	in $\graph{\cH_{t}}$.
	This implies that the graph $\graph{\mathcal{H}_{t}}$ may be 
	obtained from $\graph{\mathcal{H}_{t-1}}$ by at most 
	$(4t\ell r^2)^{\ell r^2}$ ways, and hence
	$g(t,D)\le g(t-1,D)\cdot (4t\ell r^2)^{\ell r^2}$.
	As $g(0,\emptyset,n)= 1$, we establish our lemma
	with $C = (\ell r^{2})^{4}$ by iterating the inequalities
	in~\eqref{boundg(t,D)}. 
\end{proof}

Now, we are ready to prove Lemma~\ref{lem:boundOut}:

\begin{proof}[Proof of~\thref{lem:boundOut}]
	We first claim that there exists a constant 
	$C_{1}=C_{1}(r,\ell)>0$ 
	such that $|D_{T}(\cH)| \le C_{1}$ for all 
	$\cH \in \Crit_{r,\ell}(n)$.
	Recall that $T = T(\cH)$ denotes the stopping time of 
	\hypertree$(\cH)$.
	Fix any hypergraph $\cH$ in $\Crit_{r,\ell}(n)$ and let 
	$G_{i}=\graph{\cH_{i}}$ for $i=0,\ldots,T$.
	By~\thref{deg-lambda}, we have
	$\lambda (G_{i}) \leq \lambda (G_{i-1}) - \delta$ if 
	$i \in D_{T}(\cH)$,  
	and $\lambda (G_{i}) = \lambda (G_{i-1})$ if 
	$i \not\in D_{T}(\cH)$, where $\delta = \delta(r,\ell)>0$.
	As $\lambda(G_0) = \lambda(K_r)$ 
	(by \thref{hypertreeBasics}\ref{hypertreeBasics-E0})
	and $\lambda(G_{T(\cH)-1}) > -\eps$ 
	(by~\thref{hypertreeBasics}\ref{hypertreeBasics-stopping}), 
	it follows that 
	$|D_{T}(\cH)|\le 1+(\lambda(K_{r})-\eps)/\delta$.
	As $\eps$ only depends on $r$ and $\ell$, 
	this proves our claim.

	By \thref{hypertreeBasics}\ref{hypertreeBasics-stopping}, 
	the stopping time $T$ is bounded from above by $\log n$.
	Since $|D_{T}| \le C_{1}$, the size of $\Out_{r,\ell}(n)$ 
	is bounded by the size of
	\begin{align*}
		\bigcup_{t \le \log n} \bigcup_{\substack{D \se [t]:\\ 
		|D|\le C_{1}}} \cG(t,D,n),
	\end{align*}
	where $\cG(t,D,n)$ was defined just above \thref{Fdk}.
	Using the bound on $|\cG(t,D,n)|$ given by~\thref{Fdk}, 
	we conclude that
	\begin{align*}
		|\Out_{r,\ell}(n)|\le 
		\sum \limits_{t=1}^{\lceil \log n \rceil} 
		\sum \limits_{\substack{D\se [t]:\\|D|\le C_{1}}} 
		t^{M_1 |D|}\le(\log n)^{C_{0}},
	\end{align*}
	for some $C_{0}=C_{0}(r,\ell)>0$.
\end{proof}
}

\bigskip

\noindent
{\bf Acknowledgments.}
This work was started at the thematic program GRAPHS@IMPA (January--March 2018), in Rio de Janeiro.
We thank IMPA and the organisers for the hospitality and for providing a pleasant research environment. 
We would also like to thank Rob Morris for helpful discussions.

\bibliographystyle{siam}

\appendix 

\section{Proof of Lemma~\ref{identitiesform2}}\label{appendix}

	Note that every subgraph $J \subsetneq C_{\ell}$ 
	is a forest, and so we have $e(J)\le v(J)-1$. 
	Thus, for every $J \subsetneq C_{\ell}$ with $v(J) \ge 3$ 
	this implies that 
	$(e(J)-1)/(v(J)-2) \le 1.$ 
	On the other hand, 
	\begin{align*}
		\dfrac{e(C_{\ell})-1}{v(C_{\ell})-2} = 
		\dfrac{\ell -1}{\ell-2} >
		1,
	\end{align*}
	which implies $m_2(C_{\ell}) = (\ell-1)/(\ell-2)$. 
	Now, let us analyse subgraphs of $K_{r}$.
	For each $J \se K_{r}$, we have 
	$e(J) \le \binom{v(J)}{2}$.
	Thus,
	\begin{align*}
		\dfrac{e(J)-1}{v(J)-2} \le \dfrac{\binom{v(J)}{2}-1}{v(J)-2}
		= \dfrac{v(J)+1}{2}
	\end{align*}
	for each $J \subseteq K_{r}$ such that $v(J) \ge 3$.
	It follows that $m_2(K_r) = (r+1)/2$.
	Next, for each $\ell \ge 3$, consider the function 
	$f_{\ell}: \mathbb{N} \to \mathbb{Q}$ defined by 
	$$f_{\ell}(t)=\frac{\binom{t}{2}}{t-2+m_{2}(C_{\ell})^{-1}}.$$
	It is not hard to check that $(f_{\ell}(t))_{t\ge3}$ is monotone 
	increasing (for every given $\ell$). 
	Since $m_2(C_{\ell}) = (\ell-1)(\ell-2)$, we have
	\begin{align}\label{eq:m2-f}
		m_2(K_r,C_{\ell}) = 
		f_{\ell}(r) = 
		\dfrac{\binom{r}{2}}{r-2+(\ell-2)/(\ell-1)}.
	\end{align}
	It follows readily from this identity that 
	$m_2(K_r,C_{\ell})$ is strictly decreasing in $\ell$, and thus, 
	\begin{align}\label{eq:m2-inequality}
		m_2(K_r,C_{\ell}) \le m_2(K_r,C_3) = \frac{r(r-1)}{2r-3} < 
		\frac{r+1}{2} = m_2(K_r)
	\end{align}
	for every $r \ge 4$.
	Finally, 
%	Now, it remains to show that $\mm > r/2$.
%	This can be easily seen from 
	the identity in~\eqref{eq:m2-f} implies that 
	\begin{align*}
		\mm &= \dfrac{\binom{r}{2}(\ell-1)}{(r-1)(\ell-1)-1} 
		= \dfrac{r}{2} \cdot \dfrac{1}{1-\frac{1}{(r-1)(\ell -1)}}
		> \dfrac{r}{2}. \qedhere
	\end{align*}

\end{document}